\documentclass[12pt]{amsart}
\usepackage{graphicx,amsmath,amsfonts,amsthm,amssymb,verbatim,stmaryrd,fullpage}
\newtheorem{theorem}{Theorem}
\newtheorem{lemma}[theorem]{Lemma}
\newtheorem{prop}[theorem]{Proposition}

\newtheorem{cor}[theorem]{Corollary}

\theoremstyle{remark}
\newtheorem*{remark}{Remark}

\newcommand{\R}{\mathbb R}
\newcommand{\C}{\mathbb C}
\newcommand{\Z}{\mathbb Z}
\newcommand{\Q}{\mathbb Q}

\newcommand{\A}{\mathbb A}

\newcommand{\p}{\mathfrak p}
\newcommand{\Pp}{\mathfrak P}
\newcommand{\q}{\mathfrak q}
\newcommand{\g}{\mathfrak g}

\newcommand{\n}{\mathfrak{n} }
\newcommand{\Cc}{\mathfrak{c} }

\newcommand{\Sp}{\text{Sp} }
\newcommand{\Mp}{\text{Mp} }
\newcommand{\OO}{\mathcal O}

\newcommand{\Ss}{\mathcal S}
\newcommand{\VV}{\mathcal V}

\newcommand{\End}{\text{End}}

\newcommand{\tr}{\text{tr}}

\newcommand{\Hom}{\textup{Hom}}
\newcommand{\Vol}{\text{Vol}}

\newcommand{\tG}{ \widetilde{G} }

\begin{document}

\title{Theta Lifting and Cohomology Growth in $p$-adic Towers}

\author{Mathieu Cossutta}
\email{mcossutta@gmail.com}

\author{Simon Marshall}
\address{The Institute for Advanced Study\\
Einstein Drive\\
Princeton\\
NJ 08540, USA}
\email{slm@math.princeton.edu}
\thanks{The second author is supported by NSF grant DMS-0635607.}

\begin{abstract}
We use the theta lift to study the multiplicity with which certain automorphic representations of cohomological type occur in a family of congruence covers of an arithmetic manifold.  When the family of covers is a so-called `$p$-adic congruence tower' we obtain sharp asymptotics for the number of representations which occur as lifts.  When combined with theorems on the surjectivity of the theta lift due to Howe and Li, and Bergeron, Millson and Moeglin, this allows us to verify certain cases of a conjecture of Sarnak and Xue.
\end{abstract}

\maketitle

\section{Introduction}

The purpose of this paper is to quantify results of J.-S. Li \cite{Li2} which establish the existence of certain automorphic representations of cohomological type on the classical groups.  This problem was previously considered by the first author in \cite{Co}, and this paper provides a sharp form of the results proven there.  To recall Li's results, let $F$ be a totally real field and let $(E, \iota)$ be an extension of $F$ with an involution $\iota$ of one of the following types:

\begin{eqnarray*}
E = \Bigg\{ \begin{array}{ll} F, & \text{case } 1, \\ \text{a quadratic extension } E/F, & \text{case } 2, \end{array}
\end{eqnarray*}

\begin{eqnarray*}
\iota = \Bigg\{ \begin{array}{ll} \text{id}, & \text{case } 1, \\ \text{the Galois involution of } E/F, & \text{case } 2. \end{array}
\end{eqnarray*}
We shall use $v$ and $w$ to denote places of $F$ and $E$ respectively, and denote the completion of $F$ at $v$ by $F_v$ (resp. $E$ at $w$ by $E_w$).  If $v$ is non-Archimedean, we shall say that a Hermitian space $V$ over $E$ is unramified at $v$ if $v$ does not divide 2 and the Hermitian space $V_v = V \otimes_F F_v$ over the algebra $E \otimes_F F_v$ contains a self-dual lattice $L$.  Note that this condition is automatically satisfied in the case where $v$ splits in $E$.

Let $V$ and $V'$ be finite dimensional vector spaces over $E$ with dimensions $n$ and $n'$ respectively, which we equip with non-degenerate sesqui-linear forms $( \: , )$ and $( \: , )'$.  We assume that $n' \le n$, and that $( \: , )$ is $\eta$-Hermitian and $( \: , )'$ is $-\eta$ -Hemitian for $\eta \in \{ \pm 1 \}$.  We may equip the space $W = V \otimes_E V'$ with the symplectic form

\begin{equation*}
\langle \:, \rangle = \tr_{E/F} ( ( \: , ) \otimes ( \: , )' ),
\end{equation*}
and define $G$, $G'$ and $\Sp(W)$ to be the isometry groups of $( \: , )$, $( \: , )'$ and $\langle \:, \rangle$ respectively, considered as algebraic groups over $F$.  Then $(G, G')$ is a type I irreducible dual pair inside $\Sp(W)$.  Likewise, if $G(\R)$ and $G'(\R)$ are the real points of the groups obtained after restricting scalars to $\Q$ then $(G(\R), G'(\R))$ is a type I irreducible dual pair inside the corresponding real symplectic group.  We shall assume in case 1 that both $V$ and $V'$ have even dimension; this ensures that the metaplectic cover of $\Sp(W)$ splits on restriction to $G \times G'$ by a result of Kudla \cite{K}.  Let $L \subset V$ be an $\OO_E$ lattice, whose completion in $V_v$ we shall assume to be self dual whenever $V$ is unramified, and $\Gamma \subset G(\R)$ the arithmetic group stabilising $L$.  Let $r$ be the dimension of the maximal isotropic subspace of $V$, and define

\begin{eqnarray*}
d & = & \Bigg\{ \begin{array}{ll} 0, & \text{ case 1, $\eta = 1$,} \\ 1, & \text{ case 1, $\eta = -1$,} \\ \tfrac{1}{2}, & \text{ case 2,} \end{array} \\
\delta & = & \bigg\{ \begin{array}{ll} 0, & \text{ case 1,} \\ 1, & \text{ case 2.} \end{array}
\end{eqnarray*}

In \cite{Li5}, Li characterises the representations $\pi$ of $G(\R)$ which are of cohomological type and are realised as the local theta lift of a discrete series representation from a smaller group $G'(\R)$ (which must be unique).  The set of such representations is quite large, and in particular is the entire cohomological dual in the cases where $G(\R)$ is $SO(n,1)$ or $U(n,1)$.  Moreover, in \cite{Li2} he proves that if $\pi$ is such a representation and $n > 2n' + 4d -2$ then $\pi$ is automorphic, i.e. there exists a congruence subgroup $\Gamma' \subset \Gamma$ such that $\pi$ occurs discretely in $L^2( \Gamma' \backslash G(\R) )$.  We shall state Li's results more precisely in section \ref{thetacoh}.\\

We shall state our results in adelic terms in order to avoid technicalities coming from the failure of strong approximation in orthogonal groups.  Let $\A$ be the ring of adeles of $F$, $\A_f$ the ring of finite adeles, and for $L \subset V$ a lattice which is self-dual in $V_v$ wherever possible and $\n \subset \OO_F$ an ideal, we define $K(\n)$ to be the standard family of compact subgroups of $G(\A_f)$ with respect to $L$ (notations to be defined more precisely in section \ref{thetacoh}).  If $G^0$ is the subgroup of elements with trivial determinant in $G$ and $K_\infty \subset G(\R)$ is the standard maximal compact subgroup, define

\begin{eqnarray*}
X(\n) & = & G(F) \backslash G(\A) / K(\n), \\
X^0(\n) & = & G^0(F) \backslash G^0(\A) / K^0(\n), \\
Y(\n) & = & X(\n) / K_\infty, \quad Y^0(\n) = X^0(\n) / K_\infty.
\end{eqnarray*}
As will be discussed in section \ref{thetacoh}, $X^0(\n)$ is connected in case 1 with $\eta = -1$ and case 2, and in case 1 with $\eta = 1$ has a bounded number of components if we assume that $\n = \Cc \p^k$ with $\Cc$ and $\p$ fixed.  Define $L^2_\theta ( X(\n) ) \subset L^2 ( X(\n) )$ to be the subspace generated by twists of all automorphic forms on $G$ which are in the image of the theta lift from a group $G'$ with $n'$ satisfying $n > 2n' + 4d -2$, and let $m_\theta( \pi, \n)$ be the multiplicity with which $\pi$ occurs in $L^2_\theta ( X(\n) )$.  Define $m_\theta^0( \pi, \n)$ to be the analogous multiplicity for $X^0(\n)$.  We aim to quantify Li's existence theorem by determining the asymptotic growth rate of $m_\theta^0( \pi, \n)$ for $\n$ of the form $\Cc \p^k$ as $k \rightarrow \infty$, where $\p$ is a sufficiently unramified prime which is inert in $E$ and $\pi$ is one of the cohomological forms considered by Li in \cite{Li5}.

We shall state our theorem with the aid of two sets of places $S$ and $S'$ of $F$.  We define $S$ to contain the Archimedean places, the primes above those which ramify in $E/\Q$, and the places at which the Hermitian space $V$ is ramified (in particular, all places above 2).  $S'$ may be any set of places containing $S$ and such that the $S'$-units of $F$ assume all possible combinations of sign at the Archimedean places.  Throughout the paper, we shall use $\Pp$ to denote the prime of $E$ lying above $\p$.  Our main theorem may then be stated as follows.

\begin{theorem}
\label{mainthm}
For $\p \notin S$ which is inert in $E$ and any ideal $\Cc \subset \OO_F$, we have

\begin{equation}
\label{mainupper}
m_\theta^0( \pi, \Cc \p^k ) \ll (N \Pp)^{nn'k/2}.
\end{equation}
Moreover, this upper bound is sharp after multiplying $\Cc$ by a suitable product of primes in $S'$.

\end{theorem}

\begin{remark}
It will be seen in the course of the proof that the upper bound in theorem \ref{mainthm} continues to hold in case 1 if $\eta = 1$ and $n$ is odd.  The only point at which we use our parity assumption is when we split the metaplectic cover on restriction to the dual pair, so that we may produce forms on $G(\A)$ and apply the multiplicity results of Savin \cite{Sv} in section \ref{repredn}.  In case 1 with $\eta = 1$ and $n$ odd, the Weil representation has a very simple restriction to $\tG(\A)$ so that we may still produce forms on $G(\A)$.  Moreover, the forms on $\tG'(F) \backslash \tG'(\A)$ that we are lifting still have their Archimedean components in the discrete series, and so have nonzero $(\g, K)$ cohomology.  Proposition \ref{trivbd} is therefore still valid, and the proof of the upper bound goes through as before.  The lower bound relies on the results of Savin, and we do not know whether these continue to hold for the metaplectic quotient $\tG'(F) \backslash \tG'(\A)$.
\end{remark}

\begin{remark}
It was stated in section 2.5 of \cite{Co} that the lower bounds on $m_\theta$ given there only hold under the additional assumption that $n > 5n'/2 + 1$.  This is in fact unnecessary, and its only purpose was to guarantee the stable range assumption $n > 2n' + 4d - 2$ while being simpler in form.
\end{remark}

\subsection{The Conjecture of Sarnak and Xue}

For any representation $\pi$ of $G(\R)$, let $p(\pi)$ denote the infimum over the set of $p \ge 2$ such that all $K$-finite matrix coefficients of $\pi$ lie in $L^p( G(\R) )$.  There is a conjecture of Sarnak and Xue \cite{SX} which asserts that

\begin{equation}
\label{sarnakxue}
m^0( \pi, \n ) \ll_\epsilon \Vol( X^0(\n) )^{2 / p(\pi) + \epsilon},
\end{equation}
and this may be proven in a number of cases by combining theorem \ref{mainthm} with results of Howe and Li \cite{H, Li1}, and Bergeron, Millson and Moeglin \cite{BMM} on the surjectivity of the theta lift.  If $\pi$ occurs in the Archimedean theta correspondence for the pair $(G(\R), G'(\R))$ as the lift of a discrete series representation, then one may show using theorem 3.2 of \cite{Li4} that

\begin{equation}
\label{exponent}
\frac{2}{ p(\pi) } = \frac{n'}{ n - 2 + 2d}.
\end{equation}
In comparison, theorem \ref{mainthm} may be expressed in terms of the volume of $X^0( \Cc \p^k)$ as

\begin{equation*}
m_\theta^0( \pi, \Cc \Pp^k ) \ll \Bigg\{ \begin{array}{ll} \Vol( X^0( \Cc \p^k) )^{n'/ (n-\eta) }, & \text{case } 1, \\ \Vol( X^0( \Cc \p^k) )^{nn' / (n^2-1) }, & \text{case } 2. \end{array}
\end{equation*}
It can be seen that these exponents are always strictly less than those of (\ref{exponent}), so that theorem \ref{mainthm} provides a strengthening of the Sarnak-Xue conjecture when the theta lift is surjective.  We now describe the two cases in which this is known.\\

The results of Howe and Li \cite{H, Li1} state that $m_\theta( \pi, \n) = m(\pi, \n)$ if $\pi$ occurs in the Archimedean correspondence with a representation of $G'(\R)$ with $G'(\R)$ compact and $n' < r$.  Note that this forces us to either be in case 1 with $\eta = -1$ or in case 2, so that $G$ is symplectic or unitary.  In the case of the symplectic group $\Sp_{2n}(\R)$, the representations which satisfy this condition are exactly the singular holomorphic representations of rank $< n$.  Note that the rank of a representation of $\Sp_{2n}(\R)$ must be either $n$ or an even integer $< n$; see \cite{H, Li1} for more information about these singular forms.  As an example, we may apply theorem \ref{mainthm} to the lattice $\Gamma = \Sp_{2n}(\Z)$ and its principal congruence subgroups $\Gamma(N)$ to deduce the following:

\begin{cor}
\label{singular}
If $\pi$ is a singular holomorphic representation of $\textup{Sp}_{2n}(\R)$ of rank $n' < n$ and $p \neq 2$, we have

\begin{equation*}
m( \pi, c p^k ) \ll p^{nn'k}.
\end{equation*}
Moreover, this asymptotic bound is sharp if $c$ is divisible by a sufficiently high power of 2.

\end{cor}

The singular holomorphic representations of rank $n'$ are the ones which contribute via Matushima's formula to the $L^2$ cohomology of $Y( c p^k) = \Gamma(c p^k) \backslash \Sp_{2n}(\R) / U(n)$ in bidegree $( n' (4n - n' + 2)/8, 0)$ (note that this corrects an error in theorem 0.1 of \cite{Co}).  We may therefore rephrase corollary \ref{singular} as stating that

\begin{equation*}
\dim H_{(2)}^{t(2n-t+1)/2}( Y(cp^k), \C) \ll \Vol( Y(cp^k) )^{t/(2n+1/2)}
\end{equation*}
for all $t < n/2$, and we may deduce similar results in the case of $\Sp_{2n}(\OO_F)$ for arbitrary $F$.  In the case of singular forms of rank 2 on $\Sp_{2n}(\Z)$, $m(\pi, N )$ was exactly determined by Li in \cite{Li3}.  In particular, he proves that $m(\pi, N ) = 0$ unless $4 | N$ or $p | N$ with $p \equiv -1 \; (4)$, which demonstrates that the upper bound in our Theorem is not always attained.\\

The second surjectivity Theorem that we shall apply is due to Bergeron, Millson and Moeglin \cite{BMM}, and is presently conditional on the stabilisation of the trace formula for the (disconnected) groups $GL_n \rtimes \langle \theta \rangle$ and $SO_{2n} \rtimes \langle \theta' \rangle$, where $\theta$ and $\theta'$ are the outer automorphisms.  They consider congruence arithmetic lattices in $SO(d,1)$ which arise from a nondegenerate quadratic form over a totally real number field, and prove that $m_\theta( \pi, \n) = m(\pi, \n)$ for all $\pi$ which contribute to the cohomology of $Y(\n)$ in degree $i < \lfloor d/2 \rfloor /2$.  Note that they place no restriction on the parity of $d$.  When we combine this result with Theorem \ref{mainthm} and the subsequent remark in the case of odd dimensional orthogonal spaces, this implies

\begin{cor}
Assume the results set out in section 1.18 of \cite{BMM}.  Let $Y(\n)$ be compact arithmetic hyperbolic manifolds of dimension $d$ which arise from a quadratic form over a totally real field.  If $\Cc$ and $\p$ are as in Theorem \ref{mainthm} and $i < \lfloor d/2 \rfloor /2$, we have

\begin{equation*}
\dim H^i( Y( \Cc \p^k), \C) \ll \textup{Vol}( Y( \Cc \p^k) )^{2i / d}.
\end{equation*}

\end{cor}

The structure of the paper is as follows.  We shall begin by recalling the construction of the global theta lift and its dependence on certain auxiliary data in section \ref{thetareview}.  In section \ref{thetacoh} we shall give precise statements of Li's results on the existence of cohomological forms, and review the methods used by the first author in \cite{Co} to bound their multiplicities.  Section \ref{proof} presents the modifications which we have made to the arguments of \cite{Co} in order to make them sharp.

{\bf Acknowledgements}: The authors would like to thank Nicolas Bergeron for making them aware of his work with Millson and Moeglin.  In addition, the second author would like to thank Peter Sarnak for many helpful discussions, and the Institute for Advanced Study for its generous support while carrying out this work.

This material is based upon work supported by the National Science Foundation under agreement No. DMS-6035607.  Any opinions, findings and conclusions or recommendations expressed in this material are those of the authors and do not necessarily reflect the views of the National Science Foundation.

\section{Review of the Theta Lift}
\label{thetareview}

\subsection{The Local Theta Lift}

Let $G(F_v)$, $G'(F_v)$ and $\Sp(W_v)$ denote the $F_v$-valued points of the groups $G$, $G'$ and $\Sp(W)$.  Define $\Mp(W_v)$ to be the metaplectic cover of $\Sp(W_v)$ satisfying

\begin{equation*}
1 \longrightarrow S^1 \longrightarrow \Mp(W_v) \longrightarrow \Sp(W_v) \longrightarrow 1,
\end{equation*}
and let $\tG(F_v)$ and $\tG'(F_v)$ be the inverse images of $G(F_v)$ and $G'(F_v)$ in $\Mp(W_v)$.  Because we are not in case 1 with one of $n$ and $n'$ odd, Kudla \cite{K} has proven that the covering $\tG(F_v) \times \tG'(F_v) \rightarrow G(F_v) \times G'(F_v)$ is trivial.  In particular, he has proven that one may associate a section

\begin{equation*}
i_\chi : G(F_v) \times G'(F_v) \longrightarrow \tG(F_v) \times \tG'(F_v)
\end{equation*}
to any character $\chi$ of $(E \otimes_F F_v)^\times$ whose restriction to $F_v$ is the quadratic character of the extension $E_w / F_v$.

For each choice of nontrivial additive character $\psi$ of $F_v$ there exists a Weil representation $\omega_\psi$ of $\Mp(W_v)$ in which $S^1$ acts via the standard character, see \cite{H2, W}.  We denote the restriction of $\omega_\psi$ to $G(F_v) \times G'(F_v)$ under $i_\chi$ by $\omega_\chi$; we shall discuss the dependence of this representation on $\psi$ and $\chi$ in section \ref{globthetareview}.  If $\pi_v$ is an irreducible admissible representation of $G'(F_v)$, we define

\begin{equation*}
\Theta( \pi_v, V_v) = ( \pi_v \otimes \omega_\chi )_{G'(F_v)}.
\end{equation*}
$\Theta( \pi_v, V_v)$ is a representation of $G(F_v)$, and in \cite{H2} it was conjectured by Howe that $\Theta( \pi_v, V_v)$ admits a unique irreducible quotient.  This is known in amost every case by work of Howe \cite{H2}, Waldspurger \cite{MVW} and Li \cite{Li4}.

\begin{theorem}
Howe's conjecture is true if:

\begin{itemize}

\item $v$ is Archimedean (Howe, \cite{H2}).

\item $v$ is non-Archimedean and does not divide 2 (Waldspurger, \cite{MVW}).

\item $v$ is non-Archimedean and the pair $(G'(F_v), G(F_v))$ is in the `stable range'.  If $v$ splits in $E$ this means that $n > 2n'$, and otherwise that $V_v$ contains an isotropic subspace of dimension at least $n'$ (Li, \cite{Li4}).  Both of these are guaranteed by the inequality $n > 2n' + 4d -2$.

\end{itemize}

\end{theorem}

In the cases where it is known to exist, we shall denote the unique irreducible quotient of $\Theta ( \pi_v, V_v)$ by $\theta ( \pi_v, V_v)$.

\subsection{The Global Theta Lift}
\label{globthetareview}

Let $\A$ and $\A_E$ denote the rings of adeles of $F$ and $E$ respectively.  Let $\Sp(W, \A)$, $G(\A)$ and $G'(\A)$ denote the adelic points of the groups $\Sp(W)$, $G$ and $G'$, and define $\Mp(W, \A)$ to be the metaplectic cover of $\Sp(W, \A)$ satisfying

\begin{equation*}
1 \longrightarrow S^1 \longrightarrow \Mp(W, \A) \longrightarrow \Sp(W, \A) \longrightarrow 1.
\end{equation*}
The restriction of this covering to $\Sp(W, F)$ has a unique section $i_F$, and we shall frequently consider $\Sp(W, F)$ as a subgroup of $\Mp(W, \A)$ by means of this section.  Let $\tG(\A)$ and $\tG'(\A)$ be the inverse images of $G(\A)$ and $G'(\A)$ in $\Mp(W, \A)$.  As in the local case, we may use the results of Kudla to associate a section

\begin{equation}
\label{Kudlasection}
i_\chi : G(\A) \times G'(\A) \longrightarrow \tG(\A) \times \tG'(\A)
\end{equation}
to any character $\chi$ of $\A_E^\times / E^\times$ whose restriction to $\A^\times$ is the quadratic character of the extension $E/F$.  We digress to prove the following compatability result for $i_\chi$ and $i_F$, which allows us to state Theorem \ref{mainthm} in a simple form.

\begin{lemma}
\label{agreement}
The sections $i_\chi$ and $i_F$ agree on $G(F) \times G'(F)$.
\end{lemma}

\begin{proof}

We shall adopt the notation of Kudla \cite{K} during the course of the proof.  We begin by showing that the Weil index is globally trivial, which is to say that if $a \in F$ and $\eta$ is an additive adele class character, and $\gamma_v( a_v, \eta_v)$ is the local Weil index of $a_v$ and $\eta_v$ with respect to the field $F_v$, then

\begin{equation*}
\gamma_F(a, \eta) := \prod_v \gamma_v( a_v, \eta_v) = 1.
\end{equation*}
We shall prove this by studying the metaplectic cocycle of $\Sp_2(\A)$, following Ranga Rao \cite{RR} and Kudla.

Let $Y \subset F^2$ be a one dimensional subspace, and let $c_Y$ be the metaplectic cocycle of $\Sp_2(\A)$ which is the product of the local cocycles constructed by Ranga Rao using the completions of $Y$.  We use this cocycle to identify $\Mp(W, \A)$ with $\Sp(W,\A) \times S^1$.  Choose a basis $\langle v_1, v_2 \rangle$ for $F^2$ with $v_1 \in Y$.  If $g \in \Sp_2(F)$, express $g$ with respect to the basis $\langle v_1, v_2 \rangle$ as

\begin{equation*}
g = \left( \begin{array}{cc} a & b \\ c & d \end{array} \right),
\end{equation*}
and define $x(g)$ to be $c$ if $c \neq 0$ and $d$ otherwise, and $j(g)$ to be 0 if $c = 0$ and 1 otherwise.  If $\psi$ is the adele class character used to define the global Weil representation, let $\eta$ be related to $\psi$ by $\eta(x) = \psi(x/2)$.  We then define the function $\beta_V : \Sp_2(F) \rightarrow \C^\times$ by

\begin{equation*}
\beta_V(g) = \gamma_F( x(g), \eta)^{-1} \gamma_F( \eta)^{-j(g)}.
\end{equation*}
By applying Theorem 3.1 of \cite{K} in case $1_+$ with $V = F$ and $W = F^2$, and observing that the normalised cocycle $c_Y^0$ appearing there is globally trivial by Hilbert's reciprocity law and the global triviality of the Hasse invariant, we see that

\begin{eqnarray}
\notag
c_Y(g_1, g_2) & = & \partial \beta_V(g_1, g_2), \quad g_1, g_2 \in \Sp_2(F) \\
\label{ratsplit}
& = & \beta_V( g_1 g_2) \beta_V(g_1)^{-1} \beta_V(g_2)^{-1}.
\end{eqnarray}
It follows that the unique section $\Sp_2(F) \rightarrow \Sp_2(\A)$ is given by $g \mapsto (g, \beta_V(g))$.\\

Now let $Y' = Y \alpha$ be any other one dimensional subspace, where $\alpha \in \Sp_2(F)$, and let $c_Y'$ be the associated adelic cocycle.  It is proven in lemma 4.2 of \cite{K} that $c_Y' = c_Y \partial \lambda$, where $\lambda$ is the 1-cochain given by

\begin{equation*}
\lambda(g) = c_Y( \alpha, g \alpha^{-1}) c_Y(g, \alpha^{-1}), \quad g \in \Sp_2(\A).
\end{equation*}
Specialising this to rational $g$ and applying (\ref{ratsplit}) gives $\lambda(g) = \beta_V(\alpha)^{-1} \beta_V( \alpha^{-1})^{-1}$, so that $\lambda$ is constant on the rational points.  Moreover, if $\beta_V'$ is defined analogously to $\beta_V$ with respect to $Y'$ then we have

\begin{equation*}
\partial \beta_V' = \partial \beta_V \partial \lambda.
\end{equation*}
$\beta_V'$ and $\beta_V \lambda$ therefore differ by a character of $\Sp_2(F)$, and because $\Sp_2(F)$ admits no nontrivial characters we in fact have $\beta_V' = \beta_V \lambda$. Evaluating both expressions at the identity matrix gives $\lambda = 1$, so $\beta_V = \beta_V'$.  We may easily find $g \in \Sp_2(F)$ and $Y, Y'$ such that $j(g) = j'(g) = 1$, and so that $x(g)$ assumes any value we choose while $x'(g) = 1$.  We therefore have

\begin{equation*}
\gamma_F( x(g), \eta)^{-1} \gamma_F( \eta) = \beta_V(g) = \beta_V'(g) = \gamma_F( x'(g), \eta)^{-1} \gamma_F( \eta) = \gamma_F( \eta)
\end{equation*}
as required.\\

We now consider an arbitrary pair $G \times G' \subset \Sp(W)$.  Let $Y \subset W$ be a rational maximal isotropic subspace, and $c_Y$ the adelic metaplectic cocycle as before.  We use $c_Y$ to identify $\Mp(W, \A)$ with $\Sp(W, \A) \times S^1$, and let $\chi$ and $i_\chi$ be as in (\ref{Kudlasection}).

In case 2, it may easily be checked using the global triviality of the Weil index that $i_\chi$ is trivial on $G(F) \times G'(F)$, so that

\begin{equation*}
i_\chi(g_1, g_2) = (g_1 g_2, 1), \quad g_1 \times g_2 \in G(F) \times G'(F).
\end{equation*}
On the other hand, by applying case $1_+$ of Theorem 3.1 of \cite{K} as before and using our observation that $c_Y = \partial \beta_V$ on $\Sp(W,F)$, we see that

\begin{equation*}
i_F(g) = (g, \gamma_F(\eta)^{-j(g)} ), \quad g \in \Sp(W,F),
\end{equation*}
where $j(g)$ is determined by the Bruhat cell of $g$ in $\Sp(W)$ and we again have $\eta = \tfrac{1}{2} \psi$.  The restriction of $i_F$ to $G(F) \times G'(F)$ must differ from $i_\chi$ by a character, and this character must be trivial because $G(F)$ and $G'(F)$ both intersect the largest Bruhat cell of $\Sp(W)$ in a Zariski open set.  Therefore $i_F = i_\chi$ as required.

In case 1, $i_F$ and $i_\chi$ must agree on the symplectic member of the dual pair because their difference is a charater, and $\Sp_{2n}(F)$ admits no nontrivial characters.  The agreement on the orthogonal member follows as above, this time using the fact that the orthogonal member is contained in the standard parabolic of $\Sp(W)$ so that $j$ is constant on it.

\end{proof}

We now continue to review the global lift.  For each choice of nontrivial additive character $\psi$ of $\A / F$ we let $\omega_\psi$ be the global Weil representation of $\Mp(W, \A)$.  If $W = X \oplus Y$ is a complete polarisation of $W$ we may realise $\omega_\psi$ on the space $L^2( X(\A) )$, in which the space of smooth vectors is the Bruhat-Schwarz space $S( X(\A) )$.  On $S(X(\A))$ there is a distribution $\theta$ defined by

\begin{equation*}
\theta(\phi) = \sum_{\xi \in X(F)} \phi(\xi)
\end{equation*}
which satisfies

\begin{equation*}
\theta( \omega_\psi(\gamma) \phi ) = \theta( \phi), \qquad \gamma \in \Sp(W,F).
\end{equation*}
For each $\psi \in S(X(\A))$ we set

\begin{equation*}
\theta_\phi( g, h) = \theta( \omega_\psi( gh) \phi).
\end{equation*}
$\theta_\phi$ is then a smooth, slowly increasing function on $G'(F) \backslash \tG'(\A) \times G(F) \backslash \tG(\A)$; see \cite{H2, W} for these facts.  Note that there is no ambiguity in the definition of $G'(F) \times G(F) \subset \tG'(\A) \times \tG(\A)$ by lemma \ref{agreement}.\\

Let $\pi$ be a cusp form contained in $L^2( G'(F) \backslash \tG'(\A) )$ which transforms under $S^1$ according to the inverse of the standard character.  Because of the splitting $i_\chi$, this is equivalent to giving a cusp form on $G'(F) \backslash G'(\A)$.  For each $f \in \pi$, we may then define

\begin{equation*}
\theta_\phi(f, h) = \int_{G'(F) \backslash \tG'(\A)} \theta_\phi( g, i_\chi(h) ) f(g) dg, \quad h \in G(\A).
\end{equation*}
In general, $\theta_\phi(f, h)$ is a smooth function of moderate growth on $G(F) \backslash G(\A)$.  However, if we assume in addition that $n > 2n' + 4d -2$ then the following theorem was proven by Li in \cite{Li2}.

\begin{theorem}
\label{LiL2}
$\theta_\phi(f, h)$ is square integrable, and nonzero for some choice of $f \in \pi$ and $\phi \in S(X(\A))$.  If $\Theta(\pi, V)$ is the subspace of $L^2( G(F) \backslash G(\A) )$ generated by the functions $\theta_\phi(f, h)$ under the action of $G(\A)$, then $\Theta(\pi, V)$ is an irreducible automorphic representation which is isomorphic to $\otimes_v \theta( \pi_v, V_v)$.

\end{theorem}

We shall refer to the automorphic representation $\Theta(\pi,V)$ as the global theta lift of $\pi$.  We finish this section by discussing how $\Theta(\pi,V)$ depends on the choices of $\psi$ and $\chi$ we have made.  Let $\psi_0$ be the standard additive character of the adeles of $\Q$, and choose $\psi$ to be the pullback of $\psi_0$ to $\A$ via the trace.  Any other nontrivial character of $\A / F$ is of the form $\psi(ax)$ for some $a \in F^\times$, and we denote this character by $\psi_a$.  Denote the Weil representation $\omega_{\psi_a}$ by $\omega_a$.  It is known that $\omega_a$ and $\omega_b$ have no irreducible component in common if $ab$ is not a square in $F^\times$, and that the duality correspondence between $\tG'(\A)$ and $\tG(\A)$ depends on $a$.  However, it is well known that this dependence can be transferred from $\psi_a$ to the quadratic form $( \; , )'$, in the sense that the correspondence with respect to $\omega_a$ for the pair $(G', G)$ considered as the isometry groups of $( \; , )'$ and $( \; , )$ is the same as the correspondence with respect to $\omega_\psi$ for $(G', G)$ considered as the isometry groups of $a( \; , )'$ and $( \; , )$.  Because we shall be taking a sum over all quadratic forms $( \; , )'$ in this paper, we therefore only need to consider the character $\psi$ which is unramified away from the different of $F/\Q$.

We only need to consider one choice of section $i_\chi$ for a similar reason.  Define $\A_E^1$ by

\begin{equation*}
\A_E^1 = \{ z \in \A_E | z z^\iota = 1 \}.
\end{equation*}
If $\chi_1$ and $\chi_2$ are two characters of $\A_E^\times / E^\times$ which agree on $\A^\times$, we may define a character $\alpha_{\chi_1, \chi_2}$ of $\A_E^1 / E^1$ as follows.  If $z \in \A_E^1$, we may choose $u \in \A_E$ such that $z = u^\iota / u$, and define $\alpha_{\chi_1, \chi_2}(z)$ to be

\begin{equation*}
\alpha_{\chi_1, \chi_2}(z)  = \chi_1 \chi_2^{-1}(u).
\end{equation*}
It can be seen that this is independent of $u$, and factors through $E^1$.  The sections $i_{\chi_1}$ and $i_{\chi_2}$ then differ by twisting by $\alpha_{\chi_1, \chi_2}$, i.e.

\begin{equation*}
i_{\chi_2}(g) = \alpha_{\chi_1, \chi_2}( \det g) i_{\chi_1}(g), \quad g \in G'(\A) \times G(\A).
\end{equation*}
Note that we are identifying the $S^1$ subgroup of $\Mp$ with the norm one complex numbers.  This relation between $i_{\chi_1}$ and $i_{\chi_2}$ implies that the two representations $\Theta_{\chi_1}( \pi, V)$ and $\Theta_{\chi_2}( \pi, V)$ which they provide are twists of one another.  Because we shall be considering all twists of these representations in any case, we only need to consider the lift constructed with the use of one character $\chi$, which we may take to be unramified away from the different of $E/\Q$.

\section{Construction of Cohomological Automorphic Representations}
\label{thetacoh}

This section contains a review of the results of Li in \cite{Li2, Li5} on the existence of cohomological automorphic representations on arithmetic manifolds, as well as the the approach taken by the first author toawrds making these results quantitative in \cite{Co}.

\subsection{Some Results of Li}

In \cite{Li5}, Li determines those unitary automorphic representations of $G(\R)$ which occur in the Archimedean theta correspondence for $(G', G)$ as the lift of a discrete series representation.  We recall that the cohomological unitary representations of $G(\R)$ were classified by Vogan and Zuckerman in \cite{VZ}.  They are denoted by $A_\q(\lambda)$ where $\q$ is a $\theta$-stable parabolic subalgebra of the complexified Lie algebra $\g$ of $G(\R)$ and $\lambda$ is a suitable one-dimensional representation of $\q$.  We shall state Li's results in the case where $G(\R)$ is almost simple, and allow the reader to infer the corresponding statements in the reductive case.  Choose a $\theta$-stable Levi decomposition $\q = \mathfrak{l} + \mathfrak{u}$, and set $\mathfrak{l}_0 = \mathfrak{l} \cap \g_0$, where $\g_0$ is the Lie algebra of $G(\R)$.  Li then proves

\begin{theorem}

\begin{itemize}

\item Let $\pi$ be a ``sufficiently regular'' discrete series representation of $G'(\R)$ (see \cite{Li5} for terminology).  Then $\theta( \pi, V)$ is a nonzero representation of $G(\R)$ of the form $A_\q(\lambda)$ with $\q$ and $\lambda$ explicitly determined.

\item Given an $A_\q(\lambda)$ for $G(\R)$, assume that

\begin{equation*}
\mathfrak{l}_0 = \mathfrak{l}_0' \oplus \g_1,
\end{equation*}
where $\mathfrak{l}_0'$ is a compact lie algebra and $\g_1$ is of the `same type' as $\g$.  Then there exists a group $G'$ with $n' \le n$, and a discrete series representation $\pi$ of $G'(\R)$ which is ``sufficiently regular", such that $A_\q(\lambda)$ is the theta lift of $\pi$. In particular, every $A_\q(\lambda)$ is obtained this way when $G(\R)$ is $SO(n, 1)$ or $U(n, 1)$.

\end{itemize}

\end{theorem}

We shall refer to those cohomological representations of $G(\R)$ which are obtained as the lift of a discrete series representation of $G'(\R)$ with $n > 2n' +4d -2$ as being of discrete type.  As a consequence of his proof of Theorem \ref{LiL2}, Li is able to deduce in \cite{Li2} that if $\pi$ is a cohomological representation of $G(\R)$ which is of discrete type, then it occurs as the Archimedean component of a representation in the discrete subspace of $L^2( G(F) \backslash G(\A) )$.  See section 2.2 of \cite{Co} for a description of those representations which are of discrete type, including the degrees in which they contribute to cohomology.

\subsection{Notation and Definitions}

We shall now introduce our notation and give a precise definition of $L^2_\theta(X(\n))$.  Let $\OO_v$ and $\OO_w$ denote the rings of integers of $F_v$ and $E_w$ for $v$ and $w$ finite, and define

\begin{equation*}
\widehat{\OO}_E = \prod_{ w \nmid \, \infty } \OO_w.
\end{equation*}
Let $L$ be an $\OO_E$ lattice in $V$, and let

\begin{equation*}
\widehat{L} = \prod_{ w \nmid \, \infty } L_w = L \otimes \widehat{\OO}_E
\end{equation*}
be its adelic closure.  If $\n \subset \OO_F$ is an ideal, define the principal congruence subgroup $K(\n) \subset G(\A)$ by

\begin{equation*}
K(\n) = \{ g \in G(\A_f) | (g-1) \widehat{L} \subset \n \widehat{L} \}.
\end{equation*}
We shall write $K(\n)$ as a product of the form

\begin{equation*}
K(\n) = \prod_{ v \nmid \, \infty } K_v( \n).
\end{equation*}
Similarly, if $\varpi_v$ denotes a uniformiser of $\OO_v$, then for $v$ a non-Archimedean place of $F$ and any choice of $w | v$ define

\begin{equation*}
K_v( \varpi_v^k ) = \{ g \in G(F_v) | (g-1)L_w \subset \varpi_v^k L_w \}.
\end{equation*}
We shall use $K$ and $K_v$ to denote the compact subgroup corresponding to the choice $\n = \OO_F$.  Recall that we have defined $X(\n)$ and $X^0(\n)$ to be the quotients

\begin{eqnarray*}
X(\n) & = & G(F) \backslash G(\A) / K(\n), \\
X^0(\n) & = & G^0(F) \backslash G^0(\A) / K^0(\n), \quad K^0(\n) = K(\n) \cap G^0(\A).
\end{eqnarray*}
We define $\alpha(V)$ to be

\begin{equation*}
\begin{array}{ll} n(n - \eta)/2, & \text{case } 1, \\
n^2/2, & \text{case } 2,
\end{array}
\end{equation*}
so that we have

\begin{equation*}
\text{Vol}( X(\Cc \p^k) ) \sim N\Pp^{\alpha(V) k}
\end{equation*}
(and likewise for $\alpha(V')$).\\

We define $L^2_\theta( G(F) \backslash G(\A) ) \subset L^2( G(F) \backslash G(\A) )$ to be the $G(\A)$-invariant subspace

\begin{equation*}
\bigoplus_{\alpha, V'} \bigoplus_\pi \alpha( \det g) \otimes \Theta( \pi, V).
\end{equation*}
The first sum is indexed by the $-\eta$ -Hermitian spaces $V'$ over $E$ in the stable range, and the set of characters $\alpha$ of $\mu(F) \mu(\R) \backslash \mu(\A)$ where $\mu$ is the algebraic group

\begin{equation*}
\begin{array}{ll}
\text{id}, & \text{case } 1, \eta = -1, \\
\{ \pm 1 \}, & \text{case } 1, \eta = 1, \\
U(1), & \text{case } 2.
\end{array}
\end{equation*}
The second sum is taken over the set of cuspidal automorphic representations of $G'(\A)$.  We denote the $K(\n)$-fixed subspace of $L^2_\theta( G(F) \backslash G(\A) )$ by $L^2_\theta(X(\n))$, and the set of restrictions of functions in this space to $X^0(\n)$ by $L^2_\theta( X^0(\n) )$.\\

Suppose that $A_\q(\lambda)$ is a cohomological automorphic representation of $G(\R)$ of discrete type, so that there is a quadratic space $V'_\q$ over $E \otimes_\Q \R$ and a discrete series representation $\pi_\q$ of the isometry group $U(V'_\q)$ of $V'_\q$ such that $A_\q(\lambda)$ is realised as the lift of $\pi_\q$ from $U(V'_\q)$.  We denote the multiplicity with which $A_\q(\lambda)$ occurs in $L^2_\theta(X(\n))$ (resp. $L^2_\theta(X^0(\n))$) by $m_\theta( A_\q(\lambda), \n )$ (resp. $m_\theta^0( A_\q(\lambda), \n )$).  Define $\mu(\n)$ to be the finite group

\begin{equation*}
\mu(F) \mu(\R) \backslash \mu(\A) / \det K(\n).
\end{equation*}
Because we are free to twist elements of $L^2_\theta(X(\n))$ by characters of $\mu(\n)$, and $X^0(\n)$ is equal to the kernel of all of these characters, we have

\begin{eqnarray*}
m_\theta( A_\q(\lambda), \n ) = | \mu(\n)| m_\theta^0( A_\q(\lambda), \n ),
\end{eqnarray*}
and $| \mu(\Cc \p^k)| \sim N\p^{\delta k}$.

In case 1 with $\eta = -1$ and case 2, $G^0$ satisfies strong approximation.  Therefore the locally symmetric space $Y^0(\n)$ is connected, and if $\Gamma(\n) \subset G^0(\R)$ is the set of elements which preserve $L$ and act trivially on $L / \n L$ then we have $Y^0(\n) = \Gamma(\n) \backslash G^0(\R) / K_\infty$.  In case 1 with $\eta = 1$ $G^0$ does not satisfy strong approximation, and so $Y^0(\n)$ will not be connected in general.  However, the number of connected components of $Y^0(\Cc \p^k)$  will be bounded as $k \rightarrow \infty$, and moreover any two of these components will have a common finite cover whose index above both may be bounded independently of $k$.\\

\subsection{Reduction to Local Representation Theory}
\label{repredn}

We shall assume for the rest of the paper that $\n = \Cc \p^k$ for $\Cc$ and $\p$ as in Theorem \ref{mainthm}.  The multiplicity $m_\theta( A_\q(\lambda), \Cc \p^k )$ with which $A_\q(\lambda)$ occurs in $L^2_\theta(X(\Cc \p^k))$ is equal to

\begin{eqnarray}
\notag
m_\theta( A_\q(\lambda), \Cc \p^k ) & = & \bigoplus_{ \substack{ \alpha, \\ V' \in \VV' } } \bigoplus_{ \substack{ \pi = \pi_\infty \otimes \pi_f \\ \pi_\infty \simeq \pi_\q } } \dim ( \alpha( \det g) \otimes \Theta( \pi, V)_f )^{K(\Cc \p^k)} \\
\label{mtheta}
& = & \bigoplus_{ \substack{ \alpha, \\ V' \in \VV' } } \bigoplus_{ \substack{ \pi = \pi_\infty \otimes \pi_f \\ \pi_\infty \simeq \pi_\q } } \prod_{ v \nmid \; \infty } \dim ( \alpha_v( \det g) \otimes \Theta( \pi, V)_v )^{K_v(\Cc \p^k)}.
\end{eqnarray}
Here, $\VV'$ denotes the set of Hermitian spaces $V'$ such that $V' \otimes_\Q \R \simeq V'_\q$, and $\Theta( \pi, V)_f$ denotes the finite components of the representation $\Theta( \pi, V)$.  We shall begin to prove bounds for $m_\theta$ by first controlling the Hermitian spaces $V'$ and the level of the automorphic representations $\pi$ which may contribute to the sum.  For each $V' \in \VV'$, choose an $\OO_E$ lattice $L'$ which is self-dual in $V_v'$ whenever possible, and define compact subgroups $K'(\n)$ and $K'_v(\varpi_v^k)$ of $G'(\A)$ with respect to $L'$ in the same way as with $G(\A)$.  In lemma 2.9 of \cite{Co}, the first author proves

\begin{prop}
\label{ramcontrol1}
For all places $v \notin S$:

\begin{itemize}

\item If $(\alpha_v \otimes \Theta( \pi, V)_v )^{K_v(\varpi_v^k)} \neq 0$ then $\alpha_v |_{\det K_v(\varpi_v^k)} = 1$.

\item $\Theta( \pi, V)_v^{K_v} \neq 0$ if and only if $V'_v$ is unramified and $\pi_v^{K'_v} \neq 0$.

\end{itemize}

\end{prop}

Proposition \ref{ramcontrol1} implies that the only spaces $V'$ which can contribute to the sum (\ref{mtheta}) must be unramified outside $S$ and the places dividing $\Cc$, and it is known that the number of these is finite.  Indeed, when $V'$ is orthogonal this follows from the simple observation that the existence of a self dual lattice at a place $v$ implies that $V'_v$ has trivial Hasse invariant and that its discriminant has even valuation.  Therefore there are only finitely many square classes which can be assumed by the discriminant of $V'$, and finitely many possibilities for its local Hasse invariants, and these uniquely determine $V'$ by the classification of quadratic forms over a number field.  The case where $V'$ is Hermitian reduces to the orthogonal case, as Hermitian forms are uniquely determined by their trace forms, and the symplectic case is trivial.

As a result, when proving the upper bound (\ref{mainupper}) we may assume that $\VV'$ consists of a single space $V'$.  We may control the level of $\pi$ at the places dividing $\Cc$ and those in $S$ using the following result, stated as Proposition 2.11 of \cite{Co}.

\begin{prop}
\label{ramcontrol2}
Let $v$ be any non-Archimedean place of $F$, and assume that

\begin{equation*}
( \alpha_v \otimes \theta_v( \pi_v, V_v) )^{K_v(\varpi_v^t)} \neq 0
\end{equation*}
for some $t$.  Then there exists $t'$ depending only on $t$ such that

\begin{equation*}
\alpha_v |_{\det K'_v( \varpi_v^{t'}) } = 1, \qquad \pi_v^{K'_v( \varpi_v^{t'})} \neq 0.
\end{equation*}
Moreover, there exists $C( L, L', v, t)$ such that

\begin{equation*}
\dim ( \alpha_v \otimes \theta_v( \pi_v, V_v) )^{K_v(\varpi_v^t)} \le C( L, L', v, t) \dim \pi_v^{K'_v( \varpi_v^{t'})}.
\end{equation*}

\end{prop}

(Note that the result stated in \cite{Co} does not include the statement about the ramification of $\alpha_v$, but this follows easily in the same way as the first point of Proposition \ref{ramcontrol1}.)  Propositions \ref{ramcontrol1} and \ref{ramcontrol2} together imply that the conductor of the character $\alpha$ must be bounded away from $\p$, and so the number of characters which may contribute to the sum is on the order of $N\p^{\delta k}$.  Furthermore, there exists an ideal $\Cc'$ of $\OO_F$ which is divisible only by primes in $S$ and those dividing $\Cc$ such that

\begin{eqnarray*}
m_\theta( A_\q(\lambda), \Cc \p^k ) & \ll & \bigoplus_{ \alpha } \bigoplus_{ \substack{ \pi = \pi_\infty \otimes \pi_f \\ \pi_\infty \simeq \pi_\q } } ( \alpha_\p \otimes \theta( \pi_\p, V_\p) )^{K_\p( \p^k )} \prod_{ v \nmid \; \infty \p } \dim ( \pi_v )^{K'_v(\Cc')} \\
& \ll & N\p^{\delta k} \bigoplus_{ \substack{ \pi = \pi_\infty \otimes \pi_f \\ \pi_\infty \simeq \pi_\q } } \theta_\p( \pi_\p, V_\p)^{K_\p( \p^k )} \prod_{ v \nmid \; \infty \p } \dim ( \pi_v )^{K'_v(\Cc')}.
\end{eqnarray*}

Factorise $K'(\Cc')$ as $K'_\p(\Cc') K'^\p(\Cc')$, and let $L^2_{\text{cusp}} ( X'( \Cc' \p^\infty ) )$ be the space of $K'^\p(\Cc')$-fixed vectors in $L^2_{\text{cusp}} ( G'(F) \backslash G'(\A) )$.  This space decomposes discretely with finite multiplicities under the action of $G'(\R) \times G'(F_\p)$, and we write the $\pi_\q$-isotypic component of this space as

\begin{equation*}
\pi_\q \otimes \bigoplus_{i=1}^\infty \pi_{\p, i} = \pi_\q \otimes \Pi_\p.
\end{equation*}
We may now reformulate the problem of bounding $m_\theta$ in local terms.  It follows from the trace formula (see Savin \cite{Sv}) that

\begin{equation}
\label{locasymp1}
\dim \Pi_\p^{K_\p( \p^k)} \sim N\Pp^{k \alpha(V') }.
\end{equation}
In fact, we shall use a slight refinement of this upper bound, proven in section \ref{proof}.  By studying the local theta correspondence, we shall prove that (\ref{locasymp1}) implies

\begin{equation}
\label{locasymp2}
\sum_{i=1}^\infty \dim ( \theta( \pi_{\p, i}, V_\p) )^{K_\p( \p^k)} \sim N \Pp^{nn'k/2},
\end{equation}
from which the required upper bound on $m_\theta^0$ follows immediately.\\

We will approach the lower bound in the same way, after first choosing a suitable Hermitian space $V'$.  In case 1 with $\eta = 1$ we may take the standard symplectic form, while in the other cases we choose the diagonal form

\begin{equation*}
( \; , )' = \sum_{i=1}^{n'} a_i x_i x_i^\iota
\end{equation*}
where $a_i$ are $S'$-units with the required signs at Archimedean places.  It can be seen that $V'$ is then unramified away from $S'$.  We let $L' \subset V'$ be a lattice which is self dual away from $S'$, which we use to define families of compact subgroups of $G'(\A)$.  We may also reduce proving a lower bound for $m_\theta$ to a local problem using the following counterpart of Proposition \ref{ramcontrol2}, stated in \cite{Co} as Proposition 2.13.

\begin{prop}
\label{ramcontrol3}
For every non-Archimedean place $v$ and compact open subgroup $K_v'$, there exists a $t$ such that

\begin{equation*}
\dim \theta( \pi_v, V_v)^{K_v(\varpi_v^t)} \ge \dim \pi_v^{K'_v}.
\end{equation*}

\end{prop}

As in proving the upper bound, Propositions \ref{ramcontrol1} and \ref{ramcontrol3} imply that there exists an ideal $\Cc'$ divisible only by the primes in $S'$ such that

\begin{equation*}
m_\theta( A_\q(\lambda), \Cc' \p^k ) \gg N\p^{\delta k} \bigoplus_{ \substack{ \pi = \pi_\infty \otimes \pi_f \\ \pi_\infty \simeq \pi_\q } } \theta( \pi_\p, V_\p)^{K_\p( \p^k )} \prod_{ v \nmid \, \infty \p } \dim ( \pi_v )^{K'_v},
\end{equation*}
and hence the lower bound in Theorem \ref{mainthm} would also follow from the asymptotic (\ref{locasymp2}).  Note that that factor of $N\p^{\delta k}$ again appears because of our freedom to twist by any character with trivial Archimedean type and which is only ramified at $\p$.

\section{Application of the Local Theta Correspondence}
\label{proof}

We have reduced the proof of Theorem \ref{mainthm} to showing that an asymptotic such as (\ref{locasymp1}) for the dimension of the $K_\p(\p^k)$-fixed subspace of a direct sum of irreducible admissible representations of $G'(F_\p)$ implies the corresponding asymptotic (\ref{locasymp2}) for their set of local theta lifts.  We shall in fact use a slightly refined form of the upper bound in (\ref{locasymp1}), which takes into account the decomposition of $\Pi_\p$ into irreducibles under the action of $K'_\p$.

\begin{lemma}
\label{trivbd}
If $\rho$ is an irreducible, finite dimensional complex representation of $K'_\p$ on a vector space $F_1$, we have

\begin{equation*}
\dim \Hom_{K'_\p} ( \rho, \Pi_\p ) \ll \dim F_1.
\end{equation*}

\end{lemma}

\begin{proof}
Let $K'_\infty$ be a maximal compact subgroup of $G'(\R)$, and define $Y'(\Cc)$ to be the locally symmetric space $X'(\Cc) \backslash K'_\infty$.  The proposition will follow easily from the fact that $\pi_\q$ is discrete series, and therefore contributes to the cohomology of certain automorphic vector bundles over $Y'(\Cc')$.  More precisely if $(\tau, F_2)$ is an irreducible, finite dimensional representation of $G'(\R)$ with the same infinitessmal character as $\pi_\q$ and $q = (\dim G'(\R) - \dim K'_\infty)/2$, it is well known (see chapter 2, section 5 of \cite{BW}) that

\begin{equation}
\label{gk}
H^q (\g, K; \pi_\q \otimes F_2) \simeq \C.
\end{equation}

Let $\widetilde{\mathcal{F}}$ be the trivial vector bundle over $G'(\A)$ with fibre $F_1 \otimes F_2$.  Let $G'(F)$ act on $\widetilde{\mathcal{F}}$ from the left by

\begin{equation*}
\gamma ( g, v_1 \otimes v_2 ) = ( \gamma g, v_1 \otimes \tau(\gamma) v_2 ),
\end{equation*}
where $\tau(\gamma)$ denotes the composition of $\tau$ with the natural embedding of $G'(F)$ into $G'(\R)$.  Let $K'_\p$ act on $\widetilde{\mathcal{F}}$ from the right by

\begin{equation*}
( g, v_1 \otimes v_2 )k = ( g k, \rho(k^{-1})v_1 \otimes v_2 ),
\end{equation*}
and let $K'^\p$ and $K'_\infty$ act on the right through the first factor alone.  Define $\mathcal{F}$ to be the quotient of $\widetilde{\mathcal{F}}$ by the left and right actions of $G'(F)$ and $K'_\infty K'$; it is a flat bundle over $Y'(\Cc')$, and it follows from (\ref{gk}) that

\begin{equation*}
\dim H^q_{\text{cusp}} ( Y'(\Cc'), \mathcal{F} ) = \dim \Hom_{K'_\p} ( \rho, \Pi_\p ).
\end{equation*}
Because the cuspidal cohomology is a subspace of the total cohomology, it therefore suffices to prove that

\begin{equation*}
\dim H^q ( Y'(\Cc'), \mathcal{F} ) \ll \dim F_1.
\end{equation*}
If we fix a triangulation of $Y'(\Cc')$ in which to compute $H^q ( Y'(\Cc'), \mathcal{F} )$, it is clear that the dimension of the space of chains valued in $F_1 \otimes F_2$ is $\ll \dim F_1$, from which the proposition follows.

\end{proof}

\subsection{The Upper Bound}

To deduce the upper bound in (\ref{locasymp2}) from lemma \ref{trivbd}, we shall use the lattice model of the oscillator representation of $\Mp(W_\p)$.  Note that from now on all groups and vector spaces will be assumed to be local at $\p$, and so we shall drop the subscript from our notation and also use $\p$ to denote a uniformiser of $F_\p$.  We recall that we have chosen self-dual lattices $L' \subset V'$ and $L \subset V$, and define $J \subset W$ to be the self-dual lattice $L' \otimes L$.  Let $S(W)$ be the lattice model of the Weil representation of $\Mp(W)$, consisting of compactly supported functions on $W$ which satisfy

\begin{equation*}
f(w+j) = \psi( -\tfrac{1}{2} \langle j,w \rangle f(w) ), \quad \forall j \in J.
\end{equation*}
Here, $\psi$ is the local component at $\p$ of the additive character introduced in section \ref{thetareview}, which we have assumed to be unramified.  If $H \subset \Sp(W)$ is the stabiliser of $J$, the action of $H$ in $S(W)$ may be written (up to a character) as

\begin{equation}
\label{hact}
\omega(h) \phi(x) = \phi( h^{-1} x).
\end{equation}
Let $S(W)[r]$ be the subspace of functions which are supported on $\p^{-r} J$, which is fixed by $K(\p^{2r}) \times K'(\p^{2r})$ by virtue of (\ref{hact}).  By the definition of the theta correspondence, if $\pi$ is an irreducible admissible representation of $G'$ then there is a $G'$-covariant surjection

\begin{equation*}
\varphi: \pi \otimes S(W) \rightarrow \theta(\pi, V).
\end{equation*}
Moreover, it is proven in chapter 5 of \cite{MVW} that there is a $K'$-covariant surjection

\begin{equation*}
\pi^{K'(\p^k)} \otimes S_0(W) \rightarrow \theta(\pi, V)^{K(\p^k)},
\end{equation*}
where $S_0(W) = S(W)[\tfrac{k}{2}]$ or $S(W)[\tfrac{k+1}{2}]$ depending on the parity of $k$.\\

Write the decompositions of $\Pi_\p^{K(\p^k)}$ and $S_0(W)$ into irreducible representations of $K'$ as

\begin{eqnarray*}
\Pi_\p & = & \sum_\rho \dim \Hom_{K'}( \rho, \Pi_\p) \rho, \\
S_0(W) & = & \sum_\rho n( \rho, k) \rho.
\end{eqnarray*}
Because $\varphi$ was $K'$-covariant, we have

\begin{eqnarray*}
\dim \theta(\pi_{\p,i}, V)^{K(\p^k)} & \ll & \sum_\rho n( \rho, k) \dim \Hom_{K'}( \rho, \pi_{\p, i} ) \\
\sum_{i=1}^\infty \dim \theta(\pi_{\p,i}, V)^{K(\p^k)} & \ll & \sum_\rho n( \rho, k) \dim \Hom_{K'}( \rho, \Pi_\p ) \\
& \ll & \sum_\rho n( \rho, k) \dim \rho \\
& = & \dim S_0(W) \\
& \ll & N\Pp^{n n' k/2}.
\end{eqnarray*}
This completes the proof of the upper bound in (\ref{locasymp2}), and hence in Theorem \ref{mainthm}.\\

\subsection{The Lower Bound}

To establish the lower bound, we shall prove an injectivity result for the restriction of $\varphi$ to some subspace of $\pi \otimes \omega$ which strengthens the one given in section 5 of \cite{Co}.  Our main tool in doing so will be the local Rallis scalar product formula, which states that there is a nonzero constant $c$ such that if $v, v' \in \pi$ and $\phi, \phi' \in \omega$, we have

\begin{equation}
\label{rallis}
\langle \varphi( v \otimes \phi), \varphi( v' \otimes \phi') \rangle = c \int_{G'} \langle \omega(g) \phi, \phi' \rangle \langle \pi(g) v, v' \rangle dg.
\end{equation}
See section 4 of \cite{Co} for a discussion and proof of this formula.  We shall construct a large subspace $A \subset \omega$ which is invariant under $K'(\p^k) \times K(\p^k)$ and such that $\langle \omega(g) \phi, \phi' \rangle$ is supported on $K'(\p^k)$ for $\phi$ and $\phi'$ in $A$, after which (\ref{rallis}) will imply that $\varphi : \pi^{K'(\p^k)} \otimes A \rightarrow \theta( \pi, V)^{K(\p^k)}$ is an injection.

By our assumption that $(V', V)$ lay in the stable range, we may decompose $V$ as $V_0 \oplus V_1$, where $V_0 \perp V_1$ and $V_0 = X \oplus Y$ for $X$ and $Y$ two isotropic subspaces of dimension $n'$ in duality.  Furthermore, our assumption that $L$ was self dual implies that we can arrange for this decomposition to satisfy $L = L \cap X + L \cap Y + L \cap V_1$, and we identify $V_0$ with $V' \oplus V'^*$ in such a way that $L'$ and $L'^*$ correspond to $L \cap X$ and $L \cap Y$.  Define $W_0$ and $W_1$ to be $V' \otimes V_0$ and $V' \otimes V_1$ respectively.  If $\omega_0$ and $\omega_1$ are the Weil representations associated to the Heisenberg groups $H( W_0)$ and $H( W_1)$ then we have $\omega = \omega_0 \otimes \omega_1$, and $G'$ acts via the product of its actions in each tensor factor.  As $V' \otimes V'^* \simeq \End(V')$ is a maximal isotropic subspace of $W_0$, we may work with the Schr\"odinger model of $\omega_0$ on the Schwartz space $\Ss( \End(V') )$.  Because $G'$ stabilises $V' \otimes V'^* \subset W_0$, the action of $G'$ in $\Ss( \End(V') )$ is (up to a character) given by

\begin{equation*}
\omega(g) \phi(x) = \phi( g^{-1} \circ x ).
\end{equation*}
We shall work with $\omega_1$ in the lattice model $S(W_1)$ with respect to the lattice $L' \otimes (L \cap V_1)$.\\

Assume that $k$ is even.  Let $\mathcal{E}$ be a set of representatives for the equivalence classes of elements of $GL(L' / \p^k L' )$ under right multiplication by $K'$, and for $t \in \mathcal{E}$ let $\phi_t \in \Ss( \End(V') )$ be the characteristic function of $t$.  Suppose that $t$ and $t'$ are elements of $\mathcal{E}$, and $T$, $T'$ are elements of $\text{Supp}(\phi_t)$ and $\text{Supp}(\phi_{t'})$.  Because $T$ and $T'$ are automorphisms of $L'$, it follows that if $g T = T'$ for $g \in G'$ then $g$ must preserve $L'$ so we have $g \in K'$ .  It then follows that

\begin{equation*}
\langle \omega(g) \phi_t, \phi_{t'} \rangle = \delta_{t, t'} 1_{K'(\p^k)}(g), \quad g \in G'.
\end{equation*}
If $h \in G$ is the element which multiplies $X \simeq V'^*$ by $\p^{k/2}$ and $Y \simeq V'$ by $\p^{-k/2}$, then $\omega(h) \phi_t$ is a multiple of the characteristic function of the set $\p^{-k/2} t + \p^{k/2} \End(L')$.  We let $\phi_t^0 = \omega(h) \phi_t$.  If $t, t' \in \mathcal{E}$ and $\phi, \phi' \in \omega_1^{K'(\p^k)}$, we have

\begin{eqnarray*}
\langle \varphi( v \otimes \phi_t^0 \otimes \phi), \varphi( v' \otimes \omega(h^{-1}) \phi_{t'} \otimes \phi') \rangle & = & c \int_{G'} \langle \omega(g) \phi_t^0, \omega(h^{-1}) \phi_{t'} \rangle \langle \omega(g) \phi, \phi' \rangle \langle \pi(g) v, v' \rangle dg \\
& = & c \int_{G'} \langle \omega(g) \phi_t, \phi_{t'} \rangle \langle \omega(g) \phi, \phi' \rangle \langle \pi(g) v, v' \rangle dg  \\
& = & c \delta_{t, t'} \langle \phi, \phi' \rangle \langle v, v' \rangle.
\end{eqnarray*}
It follows from this that if we let $\{ v_i \}$ be an orthonormal basis for $\pi^{K'(\p^k)}$ and $\{ \phi'_i \}$ be an orthonormal basis for $S(W_1)[\tfrac{k}{2}]$, then the sets

\begin{equation*}
\{ \varphi( v_i \otimes \phi^0_t \otimes \phi'_i ) \} \quad \text{and} \quad \{ \varphi( v_i \otimes \omega(h^{-1})\phi_t \otimes \phi'_i ) \}
\end{equation*}
are dual to each other under the inner product on $\theta(\pi, V)$.  Consequently, if we define $A$ to be the span of the vectors $\{ \phi^0_t \otimes \phi'_i \}$ then we have an injection $\pi^{K'(\p^k)} \otimes A \rightarrow \theta( \pi, V)$, and it remains to show that $K(\p^k)$ acts trivially on $A$ to deduce that the image of this map lies in $\theta(\pi, V)^{K(\p^k)}$.\\

Because we assumed that $\phi'$ was contained in $S(W_1)[\tfrac{k}{2}]$, it suffices to show that if $S(W_0)$ is the lattice model for $\omega_0$ with respect to the lattice

\begin{equation*}
J = L' \otimes L'^* \oplus L' \otimes L'
\end{equation*}
and $I: \Ss( \End( V') ) \rightarrow S(W_0)$ is an intertwining operator between the two models, then $I(\phi_t^0)$ lies in $S(W_0)[\tfrac{k}{2}]$.  We may choose $I$ to be the map

\begin{equation*}
I(\phi)(x, y) = \psi( \tfrac{1}{2} \langle x, y \rangle ) \int_{L' \otimes L'^*} \psi( \langle z, y \rangle ) \phi( x + z) dz, \qquad x \in V' \otimes V'^*, \quad y \in V' \otimes V',
\end{equation*}
where $\langle \; , \rangle$ here denotes the pairing between $V' \otimes V'^*$ and $V' \otimes V'$ given by the tensor product of the Hermitian form $( \; , )$ in the first factor and the natural pairing in the second factor.  It follows easily from our description of $\phi_t^0$ given earlier that $I(\phi)(x, y)$ vanishes if $(x,y)$ lies outside $\p^{-k/2} J$, and so $\phi^0_t \otimes \phi'_i$ is fixed by $K'(\p^k)$.  We therefore have

\begin{eqnarray*}
\dim \theta(\pi, V)^{K(\p^k)} & \gg & \dim A \times \dim \pi^{K'(\p^k)} \\
& \gg & | \mathcal{E} | \dim S(W_1)[\tfrac{k}{2}] \times \dim \pi^{K'(\p^k)} \\
& \gg & N\Pp^{( n'^2 - \alpha(V') ) k } N\Pp^{n'(n-2n')k/2} \dim \pi^{K'(\p^k)} \\
& \gg & N\Pp^{ ( nn'/2 - \alpha(V') ) k } \dim \pi^{K'(\p^k)}.
\end{eqnarray*}
Because this estimate is uniform in $\pi$, we may combine it with the lower bound of (\ref{locasymp1}) to deduce the lower bound in (\ref{locasymp2}) and complete the proof of Theorem \ref{mainthm}.


\begin{thebibliography}{9}

\bibitem{BMM}
N. Bergeron, J. Millson, C. Moeglin : \textit{Hodge type theorems for arithmetic manifolds associated to orthogonal groups}. arXiv:1110.3049v1.

\bibitem{BW}
A Borel, N. Wallach: \textit{Continuous cohomology, discrete subgroups, and representations of reductive groups}, Mathematical Surveys and Monographs, no. 67, American Mathematical Society (2000).

\bibitem{Co} M. Cossutta: \textit{Asymptotique des nombres de Betti des vari\'et\'es arithm\'etiques}, Duke Math. J. 150 (2009), no. 3, 443-488.

\bibitem{H} R. Howe: \textit{Automorphic forms of low rank}, in ``Non-commutative Harmonic Analysis", Springer Lecture Notes in Math. 880 (1980), 211-248.

\bibitem{H2} R. Howe: \textit{$\theta$ series and invariant theory}, Proc. Symp. Pure Math. 33, Amer. Math. Soc., Providence 1979.

\bibitem{K} S. Kudla: \textit{Splitting metaplectic covers of reductive dual pairs}, Israel J. Math. 87 (1994), no. 1-3, 361-401.

\bibitem{Li1} J. S. Li: \textit{Automorphic forms with degenerate Fourier coefficients}, Amer. J. Math. 119 (1997), no. 3, 523-578.

\bibitem{Li2} J. S. Li: \textit{Nonvanishing theorems for the cohomology of certain arithmetic quotients}, J. Reine Angew. Math. 428 (1992), 177-217.

\bibitem{Li3} J. S. Li: \textit{On the dimensions of spaces of Siegel modular forms of weight one}, Geom. Funct. Anal. 6 (1996), no. 3, 521-555.

\bibitem{Li4} J. S. Li: \textit{Singular unitary representations of the classical groups}, Invent. Math. 97 (1989), no. 2, 237-255.

\bibitem{Li5} J. S. Li: \textit{Theta lifting for unitary representations with nonzero cohomology}, Duke Math. J. 61 (1990), no. 3, 913-937.

\bibitem{MVW}
C. Moeglin, M.-F. Vigneras, J.-L. Walspurger, \textit{Correspondances de Howe sur un corps p-adique}.  Lecture notes in mathematics, 1291.  Springer-Verlag, 1987.

\bibitem{RR}
R. Ranga Rao: \textit{On some explicit formulas in the theory of the Weil representation}, Pacific J. Math. 157 (1993), 335-371.

\bibitem{SX} P. Sarnak, X. Xue: \textit{Bounds for multiplicities of automorphic representations}, Duke Math. J.  64  (1991),  no. 1, 207-227.

\bibitem{Sv} G. Savin: \textit{Limit multiplicities of cusp forms}, Invent. Math. 95 (1989), 149-159.

\bibitem{VZ} D. Vogan, G. Zuckerman: \textit{Unitary representations with nonzero cohomology}, Compositio Math. 53 (1984), no. 1, 51-90.

\bibitem{W} A. Weil: \textit{Sur la formule de Siegel dans le theorie des groupes classiques}, Acta Math. 113 (1965), 1-87.

\end{thebibliography}
\end{document}